\documentclass{amsproc}%
\usepackage{amsfonts}
\usepackage{amsmath}
\usepackage{amssymb}
\usepackage{graphicx}%
\providecommand{\U}[1]{\protect \rule{.1in}{.1in}}
\theoremstyle{plain}

\newtheorem{corollary}{Corollary}

\newtheorem{lemma}{Lemma}

\newtheorem{remark}{Remark}

\newtheorem{theorem}{Theorem}
\numberwithin{equation}{section}
\begin{document}
\title[Lifting Normals]{Lifting Normal Elements in Nonseparable Calkin Algebras}
\author{Ye Zhang}
\address{University of New Hampshire}
\email{yjg2@unh.edu}
\author{Don Hadwin}
\address{University of New Hampshire}
\email{don@unh.edu}
\urladdr{http://euclid.unh.edu/\symbol{126}don}
\author{Yanni Chen}
\address{University of New Hampshire}
\email{yet2@unh.edu}
\thanks{}
\thanks{Supported in part by a grant from the Simons Foundation}

\begin{abstract}
We use the remarkable distance estimate of Ilya Kachkovskiy and Yuri Safarov,
\cite{KS} to show that if $H$ is a nonseparable Hilbert space and
$\mathcal{K}$ is any closed ideal in $B\left(  H\right)  $ that is not the
ideal of compact operators, then any normal element of $B\left(  H\right)
/\mathcal{K}$ can be lifted to a normal element of $B\left(  H\right)  $.

\end{abstract}
\maketitle

\bigskip

Suppose $H$ is a Hilbert space with $\dim H=d\geq \aleph_{0}$. We let $B\left(
H\right)  $ denote the set of all (bounded linear) operators on $H$. For each
cardinal $m$ with $\aleph_{0}\leq m\leq d,$ we let
\[
\mathcal{F}_{m}\left(  H\right)  =\left \{  T\in B\left(  H\right)
:rankT=\dim \overline{T\left(  H\right)  }<m\right \}  ,
\]
and let $\mathcal{K}_{m}\left(  H\right)  $ be the norm closure of
$\mathcal{F}_{m}\left(  H\right)  $. The set $\left \{  \mathcal{K}_{m}\left(
H\right)  :\aleph_{0}\leq m\leq d\right \}  $ is the collection of all nonzero
proper norm-closed ideals of $B\left(  H\right)  $. The quotient C*-algebra
$\mathcal{C}_{m}\left(  H\right)  =B\left(  H\right)  /\mathcal{K}_{m}\left(
H\right)  $ is called the $m$\emph{-Calkin algebra}, and the quotient map
$\pi_{m}:B\left(  H\right)  \rightarrow \mathcal{C}_{m}\left(  H\right)  $ is
called the $m$\emph{-Calkin map.}

The properties of these Calkin algebras depends significantly on the
properties of the cardinal $m$. We say that $m$ is \emph{countably cofinal} if
$m$ is the supremum of a countable collection of smaller cardinals. It was
shown in \cite[Thm. 4.11]{H} that if $m$ is \emph{not} countably cofinal
(which holds exactly when $\mathcal{K}_{m}\left(  H\right)  =\mathcal{F}%
_{m}\left(  H\right)  $), then, for any separable unital C*-subalgebra
$\mathcal{A}$ of $\mathcal{C}_{m}\left(  H\right)  $, there is a unital $\ast
$-homomorphism $\rho:\mathcal{A}\rightarrow B\left(  H\right)  $ such that
$\pi_{m}\circ \rho$ is the identity on $\mathcal{A}$. Hence, in this case, for
any element $t$ of $\mathcal{C}_{m}\left(  H\right)  $ there is a $T\in
B\left(  H\right)  $ so that $\pi_{m}:C^{\ast}\left(  T\right)  \rightarrow
C^{\ast}\left(  t\right)  $ is a $\ast$-isomorphism sending $T$ to $t$. Hence
$t$ lifts to an element $T$ in $B\left(  H\right)  $ that shares all the
properties preserved under $\ast$-isomorphisms, i.e., being normal, subnormal,
hyponormal, isometric, or unitary.

When the cardinal $m$ is countably cofinal, the situation is not so clear cut.
When $m=\aleph_{0}$, the classical result of L. G. Brown, R. G. Douglas, and
P. A. Fillmore \cite{BDF} shows that the Fredholm index is the only
obstruction to lifting normal elements. When $m>\aleph_{0}$, it is still
possible to lift invertibles and unitaries, but the case of normality has
remained open since 1981. In this paper we prove that when $m>\aleph_{0}$ is
countably cofinal, then, for any normal element $t\in \mathcal{C}_{m}\left(
H\right)  $ there is a normal operator $T\in B\left(  H\right)  $ such that
$\pi_{m}\left(  T\right)  =t$.

A key ingredient in our work is the wonderful new theorem (the
\emph{Kachkovskiy-Safarov inequality}) of \cite{KS}, which estimates the
distance $d\left(  a,\mathcal{N}_{f}\left(  \mathcal{A}\right)  \right)  $ of
an element of $a$ in a C*-algebra $\mathcal{A}$ with real rank zero to the set
$\mathcal{N}_{f}\left(  \mathcal{A}\right)  $ of normal operators in
$\mathcal{A}$ with finite spectrum:%
\begin{equation}
d\left(  a,\mathcal{N}_{f}\left(  \mathcal{A}\right)  \right)  \leq C\left(
\left \Vert a^{\ast}a-aa^{\ast}\right \Vert ^{1/2}+d_{1}\left(  a\right)
\right)  , \tag{\#}%
\end{equation}
for some universal constant $C,$ and where
\[
d_{1}\left(  a\right)  =\sup_{\lambda \in \mathbb{C}}dist\left(  a-\lambda
,GL_{0}\left(  \mathcal{A}\right)  \right)  ,
\]
where $GL_{0}\left(  \mathcal{A}\right)  $ is the connected component of $1$
in the group $GL\left(  \mathcal{A}\right)  $ of invertible elements of
$\mathcal{A}$.

When $\mathcal{A}$ is a von Neumann algebra, $GL_{0}\left(  \mathcal{A}%
\right)  =GL\left(  \mathcal{A}\right)  $, so
\[
d_{1}\left(  a\right)  =\sup_{\lambda \in \mathbb{C}}dist\left(  a-\lambda
,GL\left(  \mathcal{A}\right)  \right)
\]
A nice formula for the $\sup dist\left(  b,GL\left(  \mathcal{A}\right)
\right)  $ is given by C. L. Olsen in \cite{O} and more general results appear
in the works of R. Bouldin \cite{B1},\cite{B2}. We will not need these
characterizations here. If $H$ is an infinite-dimensional Hilbert space, then
$H$ is isomorphic to $H\oplus H\oplus \cdots$ , so if $T\in B\left(  H\right)
$, we can identify $T^{\left(  \infty \right)  }=T\oplus T\oplus \cdots$ with an
operator in $B\left(  H\right)  $. Similarly, if $\mathcal{A}$ is an infinite
von Neumann algebra, then there is an orthogonal sequence $\left \{
P_{n}\right \}  $ of projections summing to $1$ so that each $P_{n}$ is
Murray-von Neumann equivalent to $1,$ so if $T\in \mathcal{A}$, we can still
view $T^{\left(  \infty \right)  }$ as an element of $\mathcal{A}$.

\begin{lemma}
\label{T infinity}Suppose $\mathcal{A}$ is an infinite von Neumann algebra
acting on a separable Hilbert space and $T\in \mathcal{A}$. Then

\begin{enumerate}
\item If $\mathcal{A}$ is of type III, then
\[
d_{1}\left(  T\right)  \leq \left \Vert T^{\ast}T-TT^{\ast}\right \Vert ^{1/2},
\]
and%
\[
d\left(  T,\mathcal{N}_{f}\left(  \mathcal{A}\right)  \right)  \leq
2C\left \Vert T^{\ast}T-TT^{\ast}\right \Vert ^{1/2}.
\]

\item If $\mathcal{A}$ is of type I$_{\infty}$ or type II$_{\infty}$, then
\[
d_{1}\left(  T^{\left(  \infty \right)  }\right)  \leq \left \Vert T^{\ast
}T-TT^{\ast}\right \Vert ^{1/2},
\]
and%
\[
d\left(  T^{\left(  \infty \right)  },\mathcal{N}_{f}\right)  \leq2C\left \Vert
T^{\ast}T-TT^{\ast}\right \Vert ^{1/2}.
\]

\end{enumerate}
\end{lemma}

\begin{proof}
Note that $\left \Vert T^{\ast}T-TT^{\ast}\right \Vert ^{1/2}$ is unchanged if
$T$ is replaced with $T-\lambda$, $T^{\left(  \infty \right)  }$ or with
$T^{\ast}$. Hence it will suffice to show $dist\left(  T,GL\left(
\mathcal{A}\right)  \right)  \leq \left \Vert T^{\ast}T-TT^{\ast}\right \Vert
^{1/2}$ in part $\left(  1\right)  $ and $dist\left(  T^{\left(
\infty \right)  },GL\left(  \mathcal{A}\right)  \right)  \leq \left \Vert
T^{\ast}T-TT^{\ast}\right \Vert ^{1/2}$ in part $\left(  2\right)  .$ The parts
involving the distance to $\mathcal{N}_{f}\left(  \mathcal{A}\right)  $ follow
immediately from the Kachkovskiy-Safarov inequality \cite{KS} (see (\#)
above). If both $T^{\ast}T$ and $TT^{\ast}$ are invertible, then $T$ is
invertible and the desired inequalities are trivially true. Since nothing
changes when $T$ is replaced by $T^{\ast}$, there is no harm in assuming that
$T^{\ast}T$ is not invertible. Then, for every $\varepsilon>0$, if we let
$P_{\varepsilon}$ be the spectral projection for $T^{\ast}T$ corresponding to
the interval $\left[  0,\varepsilon \right]  ,$ we have $P_{\varepsilon}\neq0$,%
\[
\left \Vert TP_{\varepsilon}\right \Vert ^{2}=\left \Vert P_{\varepsilon}T^{\ast
}TP_{\varepsilon}\right \Vert \leq \varepsilon
\]
and%
\[
\left \Vert P_{\varepsilon}T\right \Vert ^{2}=\left \Vert P_{\varepsilon}%
TT^{\ast}P_{\varepsilon}\right \Vert \leq \left \Vert P_{\varepsilon}T^{\ast
}TP_{\varepsilon}\right \Vert +\left \Vert T^{\ast}T-TT^{\ast}\right \Vert .
\]
Hence $\left \Vert T-P_{\varepsilon}^{^{\bot}}TP_{\varepsilon}^{^{\bot}%
}\right \Vert \leq \sqrt{\varepsilon}+\sqrt{\varepsilon+\left \Vert T^{\ast
}T-TT^{\ast}\right \Vert }.$

We first consider the case when $\mathcal{A}$ is an infinite \emph{factor}. If
$\mathcal{A}$ is a type $III$ factor, then the projections onto $\ker \left(
P_{\varepsilon}^{^{\bot}}TP_{\varepsilon}^{^{\bot}}\right)  $ and $\ker \left(
P_{\varepsilon}^{^{\bot}}TP_{\varepsilon}^{^{\bot}}\right)  ^{\ast}$ are
nonzero and equivalent, which implies that $P_{\varepsilon}^{^{\bot}%
}TP_{\varepsilon}^{^{\bot}}$ is a norm limit of invertible elements \cite{B2}.
Thus, for every $\varepsilon>0$%
\[
dist\left(  T,GL\left(  \mathcal{A}\right)  \right)  \leq \sqrt{\varepsilon
}+\sqrt{\varepsilon+\left \Vert T^{\ast}T-TT^{\ast}\right \Vert },
\]
which implies%
\[
dist\left(  T,GL\left(  \mathcal{A}\right)  \right)  \leq \left \Vert T^{\ast
}T-TT^{\ast}\right \Vert ^{1/2}.
\]

If $\mathcal{A}$ is a type I$_{\infty}$ or type II$_{\infty}$ factor, then the
projections onto $\ker \left(  P_{\varepsilon}^{^{\bot}}TP_{\varepsilon
}^{^{\bot}}\right)  ^{\left(  \infty \right)  }$ and $\ker \left(  \left(
P_{\varepsilon}^{^{\bot}}TP_{\varepsilon}^{^{\bot}}\right)  ^{\ast}\right)
^{\left(  \infty \right)  }$ are nonzero and have the form $Q^{\left(
\infty \right)  }$ and are equivalent, and we get%
\[
dist\left(  T^{\left(  \infty \right)  },GL\left(  \mathcal{A}\right)  \right)
\leq \left \Vert T^{\ast}T-TT^{\ast}\right \Vert ^{1/2}.
\]
Hence we have proved statements $\left(  1\right)  $ and $\left(  2\right)  $
when $\mathcal{A}$ is an infinite factor von Neumann algebra on a separable
Hilbert space.

For the general case, using the central decomposition \cite{KR}, there is a
direct integral decomposition $\mathcal{A}=\int_{\Omega}^{\oplus}%
\mathcal{A}_{\omega}d\mu \left(  \omega \right)  $, where each $\mathcal{A}%
_{\omega}$ is a factor von Neumann algebra. If $\mathcal{A}$ has type III,
then each $\mathcal{A}_{\omega}$ is a type III factor. If $T=\int_{\Omega
}^{\oplus}T_{\omega}d\mu \left(  \omega \right)  \in \mathcal{A}$ and
$\varepsilon>0$, then, using standard measurable cross-section arguments, it
is easy to measurably choose, for each $\omega \in \Omega$, an invertible
operator $S_{\omega}\in \mathcal{A}_{\omega}$ such that%
\[
\left \Vert T_{\omega}-S_{\omega}\right \Vert \leq \left \Vert T_{\omega}^{\ast
}T_{\omega}-T_{\omega}T_{\omega}^{\ast}\right \Vert ^{1/2}+\varepsilon.
\]
Note that
\[
\left \Vert S_{\omega}\right \Vert \leq \left \Vert T_{\omega}-S_{\omega
}\right \Vert +\left \Vert T_{\omega}\right \Vert \leq \left \Vert T_{\omega}%
^{\ast}T_{\omega}-T_{\omega}T_{\omega}^{\ast}\right \Vert ^{1/2}+\varepsilon
+\left \Vert T_{\omega}\right \Vert \leq
\]%
\[
\left \Vert T^{\ast}T-TT^{\ast}\right \Vert ^{1/2}+\varepsilon+\left \Vert
T\right \Vert \text{ a.e.}\left(  \mu \right)  ,
\]
so $\int_{\Omega}^{\oplus}S_{\omega}d\mu \left(  \omega \right)  $ is defined.
Although each $S_{\omega}$ is invertible, the operator $S=\int_{\Omega
}^{\oplus}S_{\omega}d\mu \left(  \omega \right)  \in \mathcal{A}$ might not be
invertible. However, each $S_{\omega}$ has a polar decomposition $U_{\omega
}\left \vert S_{\omega}\right \vert $ with $U_{\omega}$ unitary. Thus
$S=U\left \vert S\right \vert $ with $U=\int_{\Omega}^{\oplus}U_{\omega}%
d\mu \left(  \omega \right)  $ unitary. Hence, $S$ is the limit of the sequence
$\left \{  U\left(  \left \vert S\right \vert +1/n\right)  \right \}  $ of
invertible operators. Thus,
\[
dist\left(  T,GL\left(  \mathcal{A}\right)  \right)  \leq \left \Vert
T-S\right \Vert \leq \left \Vert T^{\ast}T-TT^{\ast}\right \Vert ^{1/2}%
+\varepsilon \text{.}%
\]
Since $\varepsilon>0$ was arbitrary, we have the desired result. The case when
the $\mathcal{A}_{\omega}$'s are type I$_{\infty}$ or type II$_{\infty}$
factors is handled similarly.
\end{proof}

\begin{corollary}
If $H$ is a Hilbert space and $T\in B\left(  H\right)  $ and $C^{\ast}\left(
T\right)  \cap \mathcal{K}_{\aleph_{0}}\left(  H\right)  =\left \{  0\right \}
,$ then%
\[
d_{1}\left(  T\right)  \leq \left \Vert T^{\ast}T-TT^{\ast}\right \Vert ^{1/2}.
\]

\end{corollary}

\begin{proof}
For every $0\neq A\in C^{\ast}\left(  T\right)  $ we have $rankA\geq \aleph
_{0},$ so $rankA=rankA^{\left(  \infty \right)  }.$ Thus, by \cite[Thm.
3.14]{H}, There is a sequence $\left \{  U_{n}\right \}  $ of unitary operators
in $B\left(  H\right)  $ such that $\left \Vert U_{n}^{\ast}T^{\left(
\infty \right)  }U_{n}-T\right \Vert \rightarrow0$, which means $d_{1}\left(
T\right)  =d_{1}\left(  T^{\left(  \infty \right)  }\right)  $.\bigskip
\end{proof}

We are now ready to prove our main theorem.

\begin{theorem}
Suppose $H$ is an infinite-dimensional Hilbert space and $\aleph_{0}<m\leq \dim
H$ is countably cofinal. If $t\in \mathcal{C}_{m}\left(  H\right)  $ is normal,
then there is a normal operator $T\in B\left(  H\right)  $ such that $\pi
_{m}\left(  T\right)  =t$.
\end{theorem}

\begin{proof}
We first choose an $S\in B\left(  H\right)  $ such that $\pi_{m}\left(
S\right)  =t$. We can write $H$ as a direct sum $H=\sum_{j\in J}^{\oplus}%
H_{j}$ with each dim$H_{j}=\aleph_{0}$ and with $H_{j}$ a reducing subspace
for $S$. Hence, we can write $S=\sum_{j\in J}^{\oplus}S_{j}$ with respect to
this decomposition. Since $S^{\ast}S-SS^{\ast}=\sum_{j\in J}^{\oplus}\left(
S_{j}^{\ast}S_{j}-S_{j}S_{j}^{\ast}\right)  \in \mathcal{K}_{m}\left(
H\right)  $, we see that $E=\left \{  j\in J:S_{j}^{\ast}S_{j}-S_{j}S_{j}%
^{\ast}\neq0\right \}  $ has cardinality at most $m$ and $S_{j}$ is normal when
$j\notin E$. If $CardE<m$, then $T=\sum_{j\in J}^{\oplus}S_{j}\chi_{E}\left(
j\right)  $ is normal and $\pi_{m}\left(  T\right)  =\pi_{m}\left(  S\right)
=t.$ Hence we can assume $J=E,$ so $\dim H=CardE=m$.

It follows from \cite[Cor. 3.11, Thm. 4.6]{H} that there is a unitary operator
$U\in B\left(  H\right)  $ and irreducible operators $A_{1},A_{2},\ldots$ and
cardinals $k_{1},k_{2},\ldots$ such that
\[
U^{\ast}SU-\sum \nolimits_{1\leq n<\infty}^{\oplus}A_{n}^{\left(  k_{n}\right)
}\in \mathcal{K}_{m}\left(  H\right)  .
\]
Hence we can assume that $S=\sum \nolimits_{1\leq n<\infty}^{\oplus}%
A_{n}^{\left(  k_{n}\right)  }$. Since each $A_{n}$ is irreducible, it must
act on a separable Hilbert space. If $k_{n}=m,$ then $A_{n}$ must be normal.
Hence we can write $S=N\oplus \sum \nolimits_{n\in F}^{\oplus}A_{n}^{\left(
k_{n}\right)  }$ , where $k_{n}<m$ whenever $n\in F$. If $F$ is finite,
$\pi_{m}\left(  N\oplus0\right)  =\pi_{m}\left(  S\right)  =t.$ Hence, we can
assume $F=\left \{  n_{1},n_{2},\ldots \right \}  $ with $\aleph_{0}\leq
k_{n_{1}}\leq k_{n_{2}}\leq \cdots<m$ and $m=\sup \left \{  k_{n_{j}}%
:j\in \mathbb{N}\right \}  $. It follows now that
\[
\lim_{j\rightarrow \infty}\left \Vert A_{n_{j}}^{\ast}A_{n_{j}}-A_{n_{j}%
}A_{n_{j}}^{\ast}\right \Vert =0.
\]
However, $A_{j}^{\left(  k_{n_{j}}\right)  }$ is unitarily equivalent to
$\left(  A_{j}^{\left(  \infty \right)  }\right)  ^{\left(  k_{n_{j}}\right)
}$ for each $j,$ which implies by Lemma \ref{T infinity},%
\[
\lim_{j\rightarrow \infty}dist\left(  A_{j}^{\left(  k_{n_{j}}\right)
},\mathcal{N}\right)  =0.
\]
Hence there is a sequence $\left \{  B_{j}\right \}  $ of normal operators such
that
\[
\lim_{j\rightarrow \infty}\left \Vert A_{j}^{\left(  k_{n_{j}}\right)  }%
-B_{j}\right \Vert =0.
\]
Hence $T=N\oplus \sum_{1\leq j<\infty}^{\oplus}B_{j}$ is normal and $\pi
_{m}\left(  T\right)  =\pi_{m}\left(  S\right)  =t$.
\end{proof}

\begin{remark}
If we suppose $m$ is not countably cofinal in the preceding proof, we see that
$CardE$ must be less than $m$ and the proof is complete.
\end{remark}

\end{document}